\documentclass[11pt,a4paper]{amsart}
\usepackage{newtxtext,newtxmath} 
\usepackage[dvipsnames]{xcolor} 
\usepackage[all]{xy} 
\usepackage{enumerate} 
\usepackage{hyperref,cite} 
\usepackage[para]{threeparttable} 
\usepackage{caption} 
\usepackage{float} 
\usepackage{pgf,tikz-cd} 
\usepackage{mathtools}

\newtheorem*{theorem*}{Theorem}
\newtheorem*{maintheorem}{Main Theorem}
\newtheorem{theorem}[equation]{Theorem}
\newtheorem{lemma}[equation]{Lemma}
\newtheorem{proposition}[equation]{Proposition}
\newtheorem{corollary}[equation]{Corollary}

\newtheorem{question}[equation]{Question}

\theoremstyle{definition}
\newtheorem*{principal}{Principal cases}

\newtheorem*{notation}{Notations}
\newtheorem{remark}[equation]{Remark}
\newtheorem*{ack}{Acknowledgement}
\newtheorem*{fund}{Funding}

\newcommand{\supp}{\operatorname{Supp}}
\newcommand{\pic}{\operatorname{Pic}}

\newcommand{\mult}{\operatorname{mult}}

\hypersetup{
    colorlinks   = true,    
    urlcolor     = Mahogany,    
    linkcolor    = Mahogany,    
    citecolor    = Mahogany,       
    pdftitle={Rigid affine cones over singular del Pezzo surfaces},%
    pdfauthor={In-Kyun Kim, Jaehyun Kim, and Joonyeong Won},%
    pdfkeywords={\texorpdfstring{$\mathbb{G}_{a}$-action, anticanonical polar cylinder, weighted projective space}{\mathbb{G}{a}-action, anticanonical polar cylinder, weighted projective space}},  
}

\makeatletter\@addtoreset{equation}{section} \makeatother 

\makeatletter 
\def\@seccntformat#1{\@ifundefined{#1@cntformat}%
    {\csname the#1\endcsname\quad} 
    {\csname #1@cntformat\endcsname}} 
\newcommand{\section@cntformat}{\thesection.~~~}
\makeatother

\title{Rigid affine cones over singular del Pezzo surfaces}

\author{In-Kyun Kim}

\address{ \emph{In-Kyun Kim}\newline \textnormal{June E Huh Center for Mathematical Challenges, Korea Institute for Advanced Study \newline 
\medskip 85 Hoegiro Dongdaemun-gu, Seoul 02455, Republic of Korea.\newline 
\texttt{soulcraw@kias.re.kr}}}

\author{Jaehyun Kim}

\address{ \emph{Jaehyun Kim}\newline \textnormal{Research Institute for Mathematical Sciences, Kangwon National University \newline 
\medskip College of Natural Sciences Building 5, 1 Kangwondaehak-gil, Chuncheon-si, 24341 Republic of Korea.\newline 
\texttt{kjh6691@kangwon.ac.kr}}}

\author{Joonyeong Won}

\address{ \emph{Joonyeong Won}\newline \textnormal{Department of Mathematics, Ewha Womans University \newline 
\medskip 52, Ewhayeodae-gil, Seodaemun-gu, Seoul, 03760, Republic of Korea.\newline 
\texttt{leonwon@ewha.ac.kr}}}


\begin{document}

\begin{abstract}
    We completely determine the existence of anticanonical polar cylinders in quasi-smooth log del Pezzo surfaces of index one.
\end{abstract}

\maketitle

All varieties considered in this article are assumed to be defined over an algebraically closed field of characteristic zero.


\section{Introduction}\label{sec:intro}
This paper aims to classify affine cones over certain log del Pezzo surfaces in weighted projective spaces that are rigid, meaning they do not admit any non-trivial $\mathbb{G}_a$-action. The $\mathbb{A}^1$-fibrations have been extensively studied in relation to $\mathbb{G}_a$-actions in affine algebraic geometry. A dominant morphism of algebraic varieties is called an \emph{$\mathbb{A}^1$-fibration} if the general fiber is isomorphic to $\mathbb{A}^1$. As noted in \cite[Theorem~2.2]{Dubouloz2015}, every surjective $\mathbb{A}^1$-fibration from a smooth affine surface to a smooth affine curve is the affine quotient morphism for a non-trivial $\mathbb{G}_a$-action on the surface. For a smooth complete base curve, it follows from \cite[Lemma~6.6.3]{Miyanishi1978} that a surjective $\mathbb{A}^1$-fibration on a smooth affine surface extends to a ~$\mathbb{P}^1$-fibration after completion and resolution. A general fiber of the extended fibration meets the boundary in exactly one point. Specifically, a smooth affine surface ~$V$ is called a \emph{Gizatullin surface} if it has an embedding into a smooth projective surface $X$ such that the divisor $D= X\setminus V$ forms a linear chain of smooth rational curves with transversal intersections. By \cite[Definition~5.1 and Theorem~5.2]{KPZ2017}, every Gizatullin surface other than $\mathbb{A}^1\times\mathbb{G}_m$ admits a non-trivial $\mathbb{G}_a$-action whose quotient morphism is an $\mathbb{A}^1$-fibration over $\mathbb{A}^1$. Consequently, the projective surface containing the Gizatullin surface admits a $\mathbb{G}_a$-action on some Zariski open subset. According to \cite{Miyanishi1975}, this implies that the projective surface contains an ~\emph{$\mathbb{A}^1$-cylinder}, meaning a Zariski open subset is isomorphic to $\mathbb{A}^1 \times Z$ for some affine curve $Z$. If the complement of the cylinder is supported on an effective $\mathbb{Q}$-divisor, which is $\mathbb{Q}$-linearly equivalent to an ample divisor ~$H$, then the cylinder is said to be \emph{$H$-polar}. The existence of such an ample polar cylinder implies a non-trivial ~$\mathbb{G}_a$-action on the corresponding affine cone over the projective surface as demonstrated by the following result.

\begin{theorem}[\hspace{1sp}{\cite[Corollary~3.2]{Kishimoto2013}}]\label{thm:correspondence}
    Let $X$ be a normal projective variety, and let $H$ be an ample Cartier divisor on $X$. Then $X$ contains an $H$-polar cylinder if and only if the generalized cone of the pair
    \begin{equation*}
        \mathrm{Spec}\left(\bigoplus_{m \geq 0}\mathrm{H}^0\left(X,\mathcal{O}_X(mH)\right)\right)
    \end{equation*}
    admits a non-trivial  $\mathbb{G}_a$-action.
\end{theorem}

From the perspective of Theorem \ref{thm:correspondence}, the existence of ample polar cylinders in del Pezzo surfaces has been comprehensively studied by numerous authors. For instance, the ample polar cylindricity of smooth del Pezzo surfaces is considered in \cite{Cheltsov2016, Cheltsov2017, Kishimoto2009, Kishimoto2014}. The anticanonical polar cylindricity of du Val del Pezzo surfaces is proved in \cite{Cheltsov2016s}, with further extensions to general ample divisors in \cite{Sawahara2025}. In \cite{KKW2024}, the absence of anticanonical polar cylinders in an infinite series of log del Pezzo hypersurfaces in weighted projective spaces was first established. In this context, it is natural to explore other cases in weighted projective spaces that admit anticanonical polar cylinders. 

\begin{question}\label{question}
    Is there a log del Pezzo surface in a weighted projective space that admits an anticanonical polar cylinder?
\end{question}

Meanwhile, let $S$ be a hypersurface in $\mathbb{P}(a_0, a_1, a_2, a_3)$ given by a quasi-homogeneous polynomial ~$\varphi(x, y, z, t)$ of degree $d$ with respect to the weights, $1 \leq a_0 \leq a_1 \leq a_2 \leq a_3$. The equation 
\begin{equation*}
    \varphi(x, y, z, t) = 0 \subset \mathrm{Spec}(\mathbb{C}[x, y, z, t])
\end{equation*}
defines a three-dimensional hypersurface with quasi-homogeneous singularity $(V, o)$, where $o = (0, 0, 0, 0)$. Recall that $S$ is said to be \emph{quasi-smooth} if the singularity $(V, o)$ is isolated, and ~$S$ is said to be \emph{well-formed} if it contains no codimension two singular stratum in $\mathbb{P}(a_0, a_1, a_2, a_3)$. Furthermore, the number defined by $a_0 + a_1 + a_2 + a_3 - d$ is called the \emph{index} of ~$S$, and it is known that $S$ is a del Pezzo surface if $a_0 + a_1 + a_2 + a_3-d > 0 $.\\ 
The quasi-smooth well-formed del Pezzo surfaces in $\mathbb{P}(a_0, a_1, a_2, a_3)$ are thoroughly described in \cite{Paemurru2018}. Paemurru \cite{Paemurru2018} gives an algorithm for classifying
quasi-smooth well-formed del Pezzo hypersurfaces of any fixed
index in weighted projective three-spaces.
The codimension-two complete intersection case of index one
is treated in \cite{KP2015}. For convenience, Tables~\ref{table_hypersurface} and~\ref{table_complete intersection} contain 23 hypersurface entries and 40 codimension-two complete intersection entries, respectively.\\ 
It is worth emphasizing the surface listed as \emph{No.1} in Table \ref{table_complete intersection}, which is the complete intersection of two hypersurfaces of degree $2n$ in $\mathbb{P}(1, 1, n, n, 2n-1)$. As a generalization of a smooth del Pezzo surface of degree $4$ when $n=1$, this surface shows strong potential to admit an anticanonical polar cylinder. Contrary to this expectation, we have confirmed in Theorem~\ref{thm:cyl} that it is the only case admitting an anticanonical polar cylinder when $n=1$ among the total 63 entries in Tables~\ref{table_hypersurface}, \ref{table_complete intersection}. The absence of an anticanonical polar cylinder of this surface with higher $n$ is established in Section \ref{sec:absence2}, while the absence of the other cases is discussed in Section \ref{sec:absence1}. We proceed to state our main theorem that answers Question ~\ref{question}.

\begin{maintheorem}\label{thm:main}
    For a quasi-smooth well-formed complete intersection log del Pezzo surface $X_{d_1,\dots,d_{n-2}}$ in $\mathbb{P}(a_0, \dots, a_{n})$ of index one, it admits an anticanonical polar cylinder if and only if it is a complete intersection of two quadric hypersurfaces in $\mathbb{P}^4$.
\end{maintheorem}

The rigidity in the following corollary is a consequence of the Main Theorem and \cite[Proposition~0.5 and Theorem~0.6(a)]{Kishimoto2013}.

\begin{corollary}\label{cor:rigid}
Anticanonical affine cones over quasi-smooth well-formed complete intersection log del Pezzo surfaces of index one are rigid, except for cones over smooth intersections of two quadrics in $\mathbb{P}^4$.
\end{corollary}


\section{Preliminaries}\label{sec:pre}
In this section, we recall some local inequalities relevant to our study. Let $S$ be a projective surface with at most klt singularities and let $D$ be an effective $\mathbb{Q}$-divisor on $S$ written as
\begin{equation*}
    D=\sum\limits_{i=1} ^{r} a_i D_i.
\end{equation*}

We now state the following lemma.

\begin{lemma}[\hspace{1sp}{\cite[Proposition 9.5.13]{Lazarsfeld2004a}}]\label{lem:multiplicity}
Let $\mathsf{p}$ be a smooth point of the surface $S$. If the log pair ~$(S,D)$ is not log canonical at $\mathsf{p}$, then 
\begin{equation*}
    \mult_{\mathsf{p}}(D) > 1.
\end{equation*}   
\end{lemma}

For orbifold singularities, the following analogous result holds. 

\begin{lemma}[\hspace{1sp}{\cite[Proposition 3.16]{Kollar1997}}]\label{lem:multiplicity_orbifold}
    Let $\mathsf{p}$ be a singular point of type $\frac{1}{r}(a, b)$ on the surface $S$.
    If the log pair $(S,D)$ is not log canonical at $\mathsf{p}$, then 
\begin{equation*}
    \mult_{\mathsf{p}}(D) > \frac{1}{r},
\end{equation*}
where we use the convention $\mult_{\mathsf{p}}(D)=\frac{1}{r}\mult_0(\pi^{\ast}(D))$ for the local quotient map $\pi:(\mathbb{A}^2,0)\to(S,\mathsf{p})$.
\end{lemma}

In addition, we have the following adjunction formula for orbifold singularities.

\begin{lemma}[\hspace{1sp}{\cite[Theorem 7.5]{Kollar1997}}]\label{lem:adjuction} 
    Let $\mathsf{p}$ be a singular point of type $\frac{1}{r}(a, b)$ of the surface $S$ and let~$\pi:(U,o) \to (S,\mathsf{p})$ be its smooth local cover. Suppose that $(S,D)$ is not log canonical at $\mathsf{p}$. If $D_j$ is a component of $D$ with coefficient $a_j\leq 1$ and its pullback $\pi^{\ast}(D_{j})$ is a single smooth reduced curve at $o$, then
    \begin{equation*}
        D_j \cdot (D-a_jD_j)=D_j\cdot\left(\sum_{i \neq j}a_iD_i\right)\geq \sum_{i \neq j}a_i (D_j \cdot D_i)_\mathsf{p} > \frac{1}{r}.
\end{equation*}
\end{lemma}

At a smooth point, Lemma~\ref{lem:adjuction} applies with $r=1$ and $\pi$ equal to the identity. The following lemma allows us to remove a component from the support of $D$ while preserving its $\mathbb{Q}$-linear equivalence class and non-log-canonicity at the chosen point under the hypotheses of (4).

\begin{lemma}[\hspace{1sp}{\cite[Lemma 2.2]{Cheltsov2016}}]\label{lem:convex}
    Let $T$ be an effective $\mathbb{Q}$-divisor on $S$ such that
\begin{itemize}
\item[$\bullet$] $T \sim_{\mathbb{Q}} D$ but $T \neq D$;
\item[$\bullet$] $T=\sum\limits_{i=1}^{r} b_iD_i$ for some non-negative rational numbers $b_i$'s.
\end{itemize}
For every non-negative rational number $\varepsilon$, put $D_{\varepsilon} = (1+\varepsilon)D-\varepsilon T$. Then
\begin{enumerate}
\item[(1)] $D_{\varepsilon} \sim_{\mathbb{Q}} D$ for every $\varepsilon \geq 0$;
\item[(2)] the set $\{ \varepsilon \in \mathbb{Q}_{>0} ~\vert~ D_{\varepsilon} ~\mathrm{is ~effective} \}$ attains a maximum $\mu$;
\item[(3)] the support of the divisor $D_{\mu}$ does not contain at least one component of $\supp(T)$;
\item[(4)] if $(S,T)$ is log canonical at $\mathsf{p}$ but $(S,D)$ is not log canonical at $\mathsf{p}$, then $(S,D_{\mu})$ is not log canonical at $\mathsf{p}$.
\end{enumerate}
\end{lemma}

For (4), observe that $D=(D_\mu+\mu T)/(1+\mu)$ is a convex combination of $D_\mu$ and $T$. If both $(S,D_\mu)$ and $(S,T)$ were log canonical at $\mathsf{p}$, then $(S,D)$ would also be log canonical at $\mathsf{p}$. Finally, consider the case when $S$ is a rational surface containing a $(-K_{S})$-polar cylinder $U \cong \mathbb{A}^1 \times Z$ defined by $D$. Then as shown in \eqref{eq:pencil}, the natural projection $U \to Z$ induces a rational map $\rho: S \dashrightarrow \mathbb{P}^1$, and we let $\mathscr{L}$ be the linear system corresponding to $\rho$.

\begin{lemma}[\hspace{1sp}{\cite[Lemma A.3]{Cheltsov2016}}]\label{lem:bdry}
    Assume that the base locus of $\mathscr{L}$ is not empty. Then for any effective ~$\mathbb{Q}$-divisor $H \sim_{\mathbb{Q}} -K_{S}$ with $\supp(H) \subset \supp(D)$, the log pair $(S,H)$ is not log canonical at the base point.
\end{lemma}


\section{Absence of Cylinders}\label{sec:absence1}
This section presents a detailed discussion on the absence of anticanonical polar cylinders, focusing on the role of log canonical thresholds in characterizing such properties. We begin by briefly defining the log canonical threshold and provide tables summarizing these values of quasi-smooth well-formed del Pezzo surfaces of index one in Section \ref{sec:appendix}.

\begin{notation} We summarize the notations used throughout this section.
    \begin{itemize}
        \item[$\bullet$] Let $\mathbb{P}(a_0, a_1, a_2, a_3)_{x, y, z, t}$ denote the weighted projective space with variables $x, y, z, t$ of weights $a_0, a_1, a_2, a_3$, respectively. Let $S_d$ be the quasi-smooth well-formed log del Pezzo hypersurface in $\mathbb{P}(a_0, a_1, a_2, a_3)$ defined by a quasi-homogeneous polynomial of degree $d$. These surfaces are presented in Table \ref{table_hypersurface}.
        
        \item[$\bullet$] Let $\mathbb{P}(a_0, a_1, a_2, a_3, a_4)_{x, y, z, t, w}$ be the weighted projective space with variables $x, y, z, t, w$ of weights $a_0, a_1, a_2, a_3, a_4$, respectively. Let $S_{d_1, d_2}$ be the quasi-smooth well-formed complete intersection log del Pezzo surface in $\mathbb{P}(a_0, a_1, a_2, a_3, a_4)$ defined by quasi-homogeneous polynomials
        of degrees $d_1$ and $d_2$, respectively. These surfaces are presented in Table \ref{table_complete intersection}.
        
        \item[$\bullet$] Let $D\sim_{\mathbb{Q}} -K_{S_d}$ (resp. $-K_{S_{d_1,d_2}}$) be an effective $\mathbb{Q}$-divisor on $S_d$ (resp. $S_{d_1,d_2}$).
        
        \item[$\bullet$] Let $H_x$ be the hyperplane section on $S_d$ (resp. $S_{d_1,d_2}$) defined by $x=0$.
        
        \item[$\bullet$] In $\mathbb{P}(a_0, a_1, a_2, a_3)$, the points $\mathsf{p}_{_x}, \mathsf{p}_{_y}, \mathsf{p}_{_z}$, and $\mathsf{p}_{_t}$ are $[1:0:0:0], [0:1:0:0], [0:0:1:0]$, and $[0:0:0:1]$, respectively, and in $\mathbb{P}(a_0, a_1, a_2, a_3, a_4)$ they are similarly defined, with the additional point $\mathsf{p}_{_w}=[0:0:0:0:1]$.
    \end{itemize}
\end{notation}

\subsection{Log canonical threshold} Let $X$ be a Fano variety with at most klt singularities. 
Then the existence of anticanonical polar cylinders in $X$ is deeply connected to the log canonical threshold (lct) of the log pair $(X,-K_{X})$, which is the number defined by 
\begin{equation*}
    \mathrm{lct} (X,-K_{X}) = \mathrm{sup} \left \{ 
    \lambda \in \mathbb{Q} ~\Bigg \vert 
    \begin{array}{ll} 
    &\hspace{-0.3cm} \mathrm{the ~~log ~~pair} ~~(X, \lambda D) \mathrm{~~is ~~log ~~canonical ~~for ~~every} \\
    &\hspace{-0.3cm} \mathrm{effective} ~~\mathbb{Q}\text{-}\mathrm{divisor} ~~D ~~\mathrm{on} ~~X \mathrm{~~with} ~~D\sim_{\mathbb{Q}} -K_{X}.  
    \end{array}  \right \}. 
\end{equation*}

It is noted that Fano varieties with large lct do not contain anticanonical polar cylinders. We refer to  the result in \cite{Cheltsov2021}.

\begin{theorem}[\hspace{1sp}{\cite[Theorem 1.26]{Cheltsov2021}}]\label{thm:alpha1}
    Let $X$ be a Fano variety with at most klt singularities. If the log canonical threshold of the pair $(X,-K_{X})$ is at least one, then $X$ does not contain any $(-K_X)$-polar cylinder.
\end{theorem}

In particular, the lct of quasi-smooth well-formed log del Pezzo hypersurfaces of index one has been extensively computed in \cite{Araujo2002,Cheltsov2010,Johnson2001,Cheltsov2008}. We present the lct in Table \ref{table_hypersurface}, and the lct of complete intersection log del Pezzo surfaces of index one are considered in \cite{KP2015,KW2019}, provided in Table \ref{table_complete intersection}.

\begin{theorem}[\hspace{1sp}{\cite{KP2015,KW2019}}]\label{thm:alpha2}
    Let $S_{d_1,d_2}$ be a quasi-smooth well-formed complete intersection log del Pezzo surface in a weighted projective space $\mathbb{P}(a_0, a_1, a_2, a_3, a_4)$ with index one, which is not the intersection of a linear cone with another hypersurface. If the surface is neither $S_{2n,2n}$ in $\mathbb{P}(1,1,n,n,2n-1)$ nor $S_{6,8}$ in $\mathbb{P}(1,2,3,4,5)$, then the log canonical threshold of the surface is at least one.
\end{theorem}

Furthermore, the absence of anticanonical polar cylinders in smooth del Pezzo surfaces is proved in \cite{Cheltsov2016}. They are listed as \emph{No.1, 2}, and \emph{3} in Table \ref{table_hypersurface} which are isomorphic to smooth del Pezzo surfaces of degrees $3$, $2$, and $1$, respectively.

\begin{corollary}\label{cor:absence}
    Every quasi-smooth well-formed complete intersection log del Pezzo surface of index one other than $S_{10}$, $S_{15}$, $S_{6,8}$, $S_{2n,2n}$ does not contain any anticanonical polar cylinder.
\end{corollary}

\begin{proof}
Eliminating variables that occur linearly preserves quasi-smoothness, well-formedness, and the index. By \cite[Lemma~5.7]{Esser2025}, the resulting codimension is at most two. It is positive because a weighted projective plane has index at least three. In codimension two, let $d_1\leq d_2$ be the defining degrees and let $x_0,\ldots,x_4$ be the coordinates of weights $a_0\leq\cdots\leq a_4$, respectively. If $d_1\leq a_4$, the first defining polynomial is independent of $x_4$, and quasi-smoothness forces the second to contain a pure power of $x_4$, giving $d_2\geq 2a_4$. Since the second equation can be solved for $x_4$ at any values of the other variables, quasi-smoothness also forces the first polynomial to have an isolated critical point in those four variables. In particular, it involves $x_3$, giving 
\begin{equation*}
    a_0+a_3\leq d_1.
\end{equation*}
If $a_1+a_2>d_1$, all partial derivatives except possibly that with respect to $x_0$ vanish on the plane defined by $x_0=x_1=0$, contradicting isolation. Thus $a_1+a_2\leq d_1$, and the index-one condition implies
\begin{equation*}
    2d_1\geq\sum_{i=0}^3 a_i=d_1+d_2-a_4+1\geq 2d_1+1,
\end{equation*}
a contradiction. Hence $d_1>a_4$. Outside the stated exceptions and the smooth cases of degrees
$1$, $2$, and ~$3$, the bound of lct follows from Table~\ref{table_hypersurface} in codimension one and
Theorem~\ref{thm:alpha2} in codimension two. Then Theorem~\ref{thm:alpha1} excludes anticanonical polar
cylinders. The smooth cases follow from \cite[Proposition~5.1]{Kishimoto2014} in degrees $1$, $2$ and
from \cite[Theorem~1.7]{Cheltsov2016} in degree $3$.
\end{proof}

In conclusion, Corollary \ref{cor:absence} confirms the absence of anticanonical polar cylinders for most quasi-smooth well-formed complete intersection log del Pezzo surfaces of index one, leaving four cases for detailed consideration. The following sections will address these remaining cases to complete the classification of anticanonical polar cylindricity.

\begin{principal}    
    To attain a full understanding of the anticanonical polar cylindricity of the cases not excluded by Corollary \ref{cor:absence}, it is necessary to consider four cases:
\begin{enumerate}
    \item[$\bullet$] $S_{10}$ in $\mathbb{P}(1,2,3,5)$;
    \item[$\bullet$] $S_{15}$ in $\mathbb{P}(1,3,5,7)$;
    \item[$\bullet$] $S_{6,8}$ in $\mathbb{P}(1,2,3,4,5)$;
    \item[$\bullet$] $S_{2n,2n}$ in $\mathbb{P}(1,1,n,n,2n-1)$,
    \end{enumerate}
\end{principal}

which correspond to \emph{No.4, 5} in Table \ref{table_hypersurface} and \emph{No.3, 1} in Table \ref{table_complete intersection}, respectively. The first three cases are discussed in this section, and the final case is considered in Section ~\ref{sec:absence2}.

\subsection{Additional absences}\label{subsec:absence} We prove that the surfaces $S_{10}$, $S_{15}$, and $S_{6,8}$ contain no anticanonical polar cylinders. Let $S\in\{S_{10},S_{15},S_{6,8}\}$, let $D\sim_{\mathbb{Q}}-K_S$ be an effective $\mathbb{Q}$-divisor, and let $H_x$ be the section of ~$S$ cut out by $x=0$. The main step is to show that $H_x$ is contained in the support of $D$ whenever $(S,D)$ is not log canonical.
After graded changes of coordinates, the three surfaces have the following defining equations:
\begin{itemize}
    \item[$\bullet$] $xf_9+t^2+a y^2z^2+y^5=0$ if $S=S_{10}$,
    \item[$\bullet$] $xf_{14}+y^5+c\,yzt+z^3=0$ if $S=S_{15}$,
    \item[$\bullet$] $xf_5+y(y^2+bt)+z^2+xw=xf_7+t(ey^2+t)+zw=0$ if $S=S_{6,8}$.
\end{itemize}
Here $f_i$ is weighted homogeneous of degree $i$, the coefficients $a$ and $c$ may vanish, and $be\ne 1$. For $S_{10}$, quasi-smoothness forces the coefficients of $t^2$, $y^5$, and $xz^3$ to be nonzero. Rescaling $t$ and $y$ normalizes the coefficients of $t^2$ and $y^5$ to one. If $\beta$ is the
resulting coefficient of $yzt$, the graded change $t'=t+\frac{\beta}{2}yz$ eliminates this term. For $S_{15}$, quasi-smoothness forces the coefficients of $y^5$, $z^3$, and $xt^2$ to be nonzero. Rescaling $y$ and $z$ normalizes the coefficients of $y^5$ and $z^3$ to one.\\
The section $H_x$ is reduced and irreducible in all three cases. For $S_{10}$, the polynomial $y^5+a y^2z^2$ has odd degree in $y$ and hence is not a square in $\mathbb{C}(y,z)$. For $S_{15}$, the equation of $H_x$ is primitive and linear in ~$t$ if ~$c\ne 0$, and is the irreducible binomial $y^5+z^3$ if $c=0$. For $S_{6,8}$, eliminating $w$ on $z\ne 0$ gives the integral curve $z^2+y(y^2+bt)=0$ on that chart. Since $be\ne 1$, the only point of $H_x$ with $z=0$ is $\mathsf{p}_w$. As $H_x$ is a complete
intersection of pure dimension one, it has no additional components or embedded points. On the chart $x=1$, the equations of $S_{10}$ and $S_{15}$ are quadratic in $t$ with nonzero constant leading coefficients. Thus the projection forgetting $t$ is finite on this chart. The table below summarizes the data used in the proof of Lemma~\ref{lem:lc}.

\begin{table}[H]
\centering
\caption*{Anticanonical degrees and  linear systems used in Lemma~\ref{lem:lc}}
\label{tab:lc-cases}
\begin{tabular}{cccc}
\hline
$S$ & $(-K_S)^2$ & $|\mathcal{O}_S(r)|$ \\
\hline
$S_{10}$ & $\frac{1}{3}$ & $\vert \mathcal{O}_{S_{10}}(3) \vert$ by Lemma~\ref{lem:supportM} \\
$S_{15}$ & $\frac{1}{7}$ & $\vert \mathcal{O}_{S_{15}}(5) \vert$ by Lemma~\ref{lem:supportM} \\
$S_{6,8}$ & $\frac{2}{5}$ & $\vert \mathcal{O}_{S_{6,8}}(2) \vert$ by direct analysis of a member \\
\hline
\end{tabular}
\end{table}

Notably, the existence of a cylinder is not assumed until after the proof of Lemma~\ref{lem:support}. For the proof of Lemma \ref{lem:lc}, it is worthwhile to mention a result from \cite{Araujo2002}.

\begin{lemma}[\hspace{1sp}{\cite[Corollary 3.7]{Araujo2002}}]\label{lem:supportM}
    Let $X$ be an anticanonically embedded quasi-smooth log del Pezzo surface of degree $d$ in $\mathbb{P}(a_0, a_1, a_2, a_3)$. Let $\pi_t : X \to \mathbb{P}(a_0, a_1, a_2)$ denote the projection from $\mathsf{p}_{_t} = (0: 0: 0: 1)$. Assume that $\pi_t$ has only finite fibers. If $\mathrm{H}^0(\mathbb{P}, \mathcal{O}_{\mathbb{P}}(r))$ contains
    \begin{itemize}
        \item[$\bullet$] at least two different monomials of the form $x^{i}y^{j}$,
        \item[$\bullet$] at least two different monomials of the form $x^{\ell}z^{m}$, 
    \end{itemize}
    where $r$ is a positive integer and the four constants $i,j, \ell$ and $m$ are nonnegative integers, then for every $\mathsf{p} \in X \setminus H_{x}$ and every effective $\mathbb{Q}$-divisor $D \sim_{\mathbb{Q}} -K_X$, there exists a divisor ~$F \in \vert \mathcal{O}_{X}(r) \vert$ that can be written as $F = F_0 + aH_{x}$, where $a$ is a nonnegative integer, and $F_{0}$ is an effective Weil divisor passing through $\mathsf{p}$ and has no irreducible component in common with $D$. Hence,
    \begin{equation*}
        \dfrac{r d}{a_0 a_1 a_2 a_3} \geq \mathrm{mult}_{\mathsf{p}}(D).
    \end{equation*}
\end{lemma}

As a first step, we examine the non-lc locus of the log pair $(S,D)$.

\begin{lemma}\label{lem:lc}
    The log pair $(S,D)$ is log canonical at every point of $S\setminus H_x$.
\end{lemma}

\begin{proof} As shown in Subsection~\ref{subsec:absence}, we may assume that the surfaces are defined by quasi-homogeneous polynomials as follows:
\begin{itemize}
    \item[$\bullet$] $S_{10} \subset \mathbb{P}(1,2,3,5)$ is defined by 
        \begin{equation*}
            xf_9+t^2+a y^2z^2+y^5=0,
        \end{equation*}
        where $f_{9}$ is a quasi-homogeneous polynomial of degree $9$.
    
    \item[$\bullet$] $S_{15} \subset \mathbb{P}(1,3,5,7)$ is defined by
        \begin{equation*}
            xf_{14}+y^5+c yzt+z^3=0,
        \end{equation*}
        where $f_{14}$ is a quasi-homogeneous polynomial of degree $14$.

    \item[$\bullet$] $S_{6,8} \subset \mathbb{P}(1,2,3,4,5)$ is defined by
        \begin{equation*}
            xf_5+y(y^2+bt)+z^2+xw=xf_7+t(ey^2+t)+zw=0, 
        \end{equation*}
        where $be\neq 1$ and $f_{5}, f_{7}$ are quasi-homogeneous polynomials of degree $5, 7$, respectively.
\end{itemize}

\textbf{Case 1.} Let $S\in\{S_{10},S_{15}\}$ and suppose that the log pair $(S,D)$ is not log canonical at $\mathsf{p}\in S\setminus H_x$.\\ 
For $S_{10}$, the monomials $x^3$, $xy$, and $z$ form a basis for $\vert \mathcal{O}_{S_{10}}(3) \vert$. For $S_{15}$, $x^5$, $x^2y$, and $z$ form a basis for $\vert \mathcal{O}_{S_{15}}(5) \vert$. By Subsection~\ref{subsec:absence}, the projection forgetting $t$ is finite on $S\setminus H_x$, so the pencil through $\mathsf{p}$ obtained from the listed monomials has finite base locus on this chart.
Hence, the argument of Lemma~\ref{lem:supportM} gives
\begin{equation*}
    \mult_{\mathsf{p}}(D)\leq r(-K_S)^2=
    \begin{cases}
        1 & \text{if } S=S_{10},\\[5pt]
        \frac{5}{7} & \text{if } S=S_{15}.
    \end{cases}
\end{equation*}
Since $\mathsf{p}$ is smooth, Lemma~\ref{lem:multiplicity} gives
$\mult_{\mathsf{p}}(D)>1$, a contradiction.\\

\textbf{Case 2.} Suppose that the log pair $(S_{6,8},D)$ is not log canonical at ~$\mathsf{p} \in S_{6,8} \setminus H_x$.\\ 
Then $S_{6,8}$ is smooth at $\mathsf{p}$. Let $\vert \mathcal{O}_{S_{6,8}}(2) \vert$ be the linear system with a basis of the monomials
\begin{equation*}
    x^2,\ y.
\end{equation*}
Then there exists a member $F \in \vert \mathcal{O}_{S_{6,8}}(2) \vert$ such that $\mathsf{p} \in F$ and $F$ is defined by the quasi-homogeneous polynomials
\begin{equation*}
    y-\alpha x^2 = xf_5(x,y,z,t) + y(y^2 + bt) + z^2 + xw = xf_7(x,y,z,t,w) + t(ey^2 + t) + zw= 0.
\end{equation*}
On the chart $x=1$, substituting $y=\alpha$, eliminating $w$, and completing the square give
\begin{equation*}
    t'^2=P_3(z),
\end{equation*}
where the degree of $P_{3}$ is $3$. The odd degree of $P_3$ implies irreducibility on this chart, and the
term $t'^2$ gives 
\begin{equation*} 
    \mult_{\mathsf{p}}(F)\leq 2.
\end{equation*}
Since $F$ meets $x=0$ only at $\mathsf{p}_w$, it is irreducible. This implies that the log pair $(S_{6,8},\frac{1}{2}F)$ is log canonical at ~$\mathsf{p}$. By Lemma \ref{lem:convex}, we may assume that $F \nsubset \supp(D)$. We have
\begin{equation*}
    \frac{4}{5}= F \cdot D \geq \mult_{\mathsf{p}} (D) > 1,
\end{equation*}
which is also a contradiction.
\end{proof}

\begin{lemma}{\label{lem:support}}
    If the log pair $(S,D)$ is not log canonical at $\mathsf{p}$, then $H_x$ is contained in the support of $D$.
\end{lemma}

\begin{proof}
Suppose that $H_x\nsubset\supp(D)$. Since $H_x$ is irreducible, it has no common component with $D$. By Lemma~\ref{lem:lc}, we have $\mathsf{p}\in H_x$. Since $H_x\sim_{\mathbb{Q}} D$ with $D\sim_{\mathbb{Q}}-K_S$, we have
\begin{equation*}
    H_x\cdot D=(-K_S)^2=
    \begin{cases}
        \frac{1}{3} & \text{if } S=S_{10},\\[5pt]
        \frac{1}{7} & \text{if } S=S_{15},\\[5pt]
        \frac{2}{5} & \text{if } S=S_{6,8}.
    \end{cases}
\end{equation*}

If $\mathsf{p}$ is smooth, then $\mult_{\mathsf{p}}(D)>1$ by Lemma~\ref{lem:multiplicity}. Hence, we have a contradiction
\begin{equation*}
    H_x\cdot D\geq \mult_{\mathsf{p}}(H_x)\mult_{\mathsf{p}}(D)>1.
\end{equation*}

If $\mathsf{p}$ is singular, let $\pi:(\mathbb{A}^2,0) \to (S,\mathsf{p})$ be the smooth quotient cover of degree $r$, where $r=3$, $7$, and $5$ for $S_{10}$, $S_{15}$, and $S_{6,8}$, respectively. Set $m(B)\coloneqq \mult_{0}(\pi^*(B))$. The local equation of $H_x$ gives $m(H_x) \geq 2$. Then by Lemma~\ref{lem:multiplicity_orbifold}, we have 
\begin{equation*}
    H_x\cdot D\geq (H_x\cdot D)_{\mathsf{p}} = \frac{(\pi^*(H_x)\cdot\pi^*(D))_0}{r}\geq\frac{m(H_x)m(D)}{r}>\frac{2}{r},
\end{equation*}
which again contradicts the corresponding intersection number above.
\end{proof}

From now on, we additionally assume that $S$ admits an anticanonical polar cylinder. This means that there exists an effective $\mathbb{Q}$-Cartier divisor $D \sim_{\mathbb{Q}} -K_{S}$ such that the Zariski open subset 
\begin{equation*}
    S \setminus \supp(D)
\end{equation*}
is isomorphic to $\mathbb{A}^1 \times Z$ for some affine curve $Z$. It is also said that $D$ is a boundary of the cylinder. Then we have the following commutative diagram:
\begin{equation}\label{eq:pencil}
    \begin{tikzcd}[row sep=3em, column sep=3em]
        \mathbb{A}^1 \times Z \cong U \arrow[dd, "p_2"] \arrow[rr, hook] & & S \arrow[dd, dashed, "\rho"] & & W \arrow[ll, "\psi"'] \arrow[dd, "\sigma"] \arrow[ddll, "\widetilde{\rho}"] \\
        & & & & \\
        Z \arrow[rr, hook] & & \mathbb{P}^1 & & \mathbb{F}_1 \arrow[ll]
    \end{tikzcd}
\end{equation}

Here, $p_2$ is the projection to the second factor, $\rho$ is the induced rational map with the corresponding pencil $\mathscr{L}$, and $\tilde{\rho}$ is a morphism obtained by resolving the indeterminacy of $\rho$, if any.

\begin{lemma}\label{lem:baseptfree}
    The base locus of the pencil $\mathscr{L}$ cannot be empty. 
\end{lemma}

\begin{proof}
Suppose that the linear system $\mathscr{L}$ induced by the cylinder with boundary 
\begin{equation*}
    D=\sum_{i=1}^{r} \lambda_i C_i
\end{equation*}
defines a conic bundle $\rho$ with the unique horizontal component $C_{1}$ of $D$. Then for a general fiber $F$ of the pencil, we have 
\begin{equation*}
    -2+\lambda_1 =\left(K_{S}+\sum_{i=1}^{r} \lambda_i C_i \right)\cdot F = (K_{S}+D)\cdot F = 0.
\end{equation*}

This implies that $\lambda_{1}=2$, hence the log pair $(S,D)$ is not log canonical at every point of $C_1$. By Lemma \ref{lem:lc}, 
\begin{equation*}
    C_1 \subset \supp(H_x).
\end{equation*}
Since $H_x$ is irreducible, write $H_x=\mu C_1$. We also write 
\begin{equation*}
D=2C_1+R,
\end{equation*}
where $R$ is effective and nonzero because the boundary contains a fiber over a point of $\mathbb{P}^1\setminus Z$. As $H_x\sim_{\mathbb{Q}}D$, for an ample divisor $A$ we have
\begin{equation*}
    (\mu-2)A\cdot C_1=A\cdot R>0.
\end{equation*}
Hence $\mu>2$, which contradicts the reducedness of $H_{x}$.
\end{proof}

\begin{theorem}\label{thm:absence}
    $S$ does not contain any anticanonical polar cylinder.
\end{theorem}

\begin{proof}
Suppose that $S$ admits a $(-K_{S})$-polar cylinder defined by an effective $\mathbb{Q}$-divisor
\begin{equation*}
    D=\sum_{i=1}^{r} \lambda_i C_i.
\end{equation*}

Let $\mathsf{p}$ be the base point of the corresponding linear system. Then for a resolution of indeterminacy $\psi : W \to S$ at $\mathsf{p}$, let
\begin{equation*}
    K_{W}+\sum_{i=1}^{r} \lambda_i \widetilde{C_i}+\sum_{j=1}^{m} \mu_j E_j=\psi^{\ast}(K_{S}+D),
\end{equation*}
where $E_1$ is the exceptional section in $W$ and $\widetilde{C_{i}}$ is the strict transform under $\psi$ of $C_{i}$. Then for a strict transform $\widetilde{F}$ under $\psi$ of a general pencil member $F$, we have
\begin{equation*}
    -2+\mu_1=\left( K_{W}+\sum_{i=1}^{r} \lambda_i \widetilde{C_i}+\sum_{j=1}^{m} \mu_j E_j\right)\cdot \widetilde{F}=\psi^{\ast}(K_{S}+D)\cdot\widetilde{F}=0.
\end{equation*}

This implies that the log pair $(S,D)$ is not log canonical at $\mathsf{p}$. Furthermore, the classes of the irreducible components of $D$ generate $\pic(S)_{\mathbb{Q}}$, whose rank is $8$ or $9$. Hence, $D\neq H_x$, and by Lemmas~ \ref{lem:convex}, \ref{lem:bdry}, we have an effective $\mathbb{Q}$-divisor $D' \sim_{\mathbb{Q}}-K_{S}$ such that the log pair $(S,D')$ is not log canonical at $\mathsf{p}$ and 
\begin{equation*}
    H_x \not\subset \supp(D'),
\end{equation*}
which contradicts Lemma \ref{lem:support}.
\end{proof}


\section{Cylindricity of Complete Intersection of degree 2n}\label{sec:absence2} 
Now we study the anticanonical polar cylindricity of the surface $S_{2n,2n}$ in $\mathbb{P}(1,1,n,n,2n-1)$. For $n=1$, the surface $S_{2,2}$ is a smooth del Pezzo surface of degree $4$. The proof of \cite[Theorem~3.19]{Kishimoto2009} constructs an anticanonical polar cylinder on such a surface. As a generalization of this, we construct a birational morphism from the surface $S_{2n,2n}$ to the weighted projective plane $\mathbb{P}(1,1,2n-1)$. This allows us to understand the cylindricity depending on the positive integer $n$. For later use, we recall the construction of a reducible member of the
pencil generated by $x$ and $y$ from the proof of \cite[Lemma~6.1]{KP2015}, where it is used to compute the log canonical threshold of $S_{2n,2n}$.

\begin{lemma}[\hspace{1sp}{\cite[Lemma 6.1]{KP2015}}]\label{lem:LR}
    Assume that $n\geq 2$. Let $\mathscr{C}$ be the pencil on $S_{2n,2n}$ generated by $x$ and $y$. Then $\mathscr{C}$ contains a member $C_{\alpha_1}=L_1+L'_1$, where $L_1$ and $L'_1$ are distinct
    quasi-smooth curves.
\end{lemma}

\begin{proof} 
The surface $S_{2n,2n}$ can be assumed to be defined by the quasi-homogeneous equations
\begin{equation*}
    \begin{array}{cc} 
        &wx+z(a_1z+b_1t)+f_n(x,y)z+\hat{f}_{n}(x,y)t+f_{2n}(x,y)=0, \\
        &wy+t(a_2z+b_2t)+g_n(x,y)z+\hat{g}_{n}(x,y)t+g_{2n}(x,y)=0,  
    \end{array}
\end{equation*}
where $f_n, g_n, \hat{f}_n,$ and $\hat{g}_n$ are homogeneous polynomials of degree $n$ and  $f_{2n}, g_{2n}$ are homogeneous polynomials of degree $2n$. For the surface to be quasi-smooth, the polynomials
\begin{equation*}
    a_1z+b_1t,\ a_2z+b_2t
\end{equation*}
should not be proportional and $a_1 \neq 0, b_2 \neq 0$. The surface has a unique singular point $\mathsf{p_{_w}}$. Let $C_{\alpha}$ be the member of the pencil cut out by the equation $x-\alpha y = 0$ on the surface. Then ~$C_{\alpha}$ is defined by the quasi-homogeneous equations 
\begin{equation*}
    \begin{array}{ll} 
            &\alpha wy+z(a_1z+b_1t)+f_n(\alpha,1)y^{n}z+\hat{f}_{n}(\alpha,1)y^{n}t+f_{2n}(\alpha,1)y^{2n}=0, \\
            &wy+t(a_2z+b_2t)+g_n(\alpha,1)y^{n}z+\hat{g}_{n}(\alpha,1)y^{n}t+g_{2n}(\alpha,1)y^{2n}=0  
    \end{array} 
\end{equation*}
in $\mathrm{Proj}(\mathbb{C}[y,z,t,w])$. Consider the affine piece of the curve $C_\alpha$ defined by $y \neq 0$. It is the affine curve defined by the equations
\begin{equation*}
    \begin{array}{cc} 
        &\alpha w+z(a_1z+b_1t)+f_n(\alpha,1)z+\hat{f}_{n}(\alpha,1)t+f_{2n}(\alpha,1)=0, \\
        &w+t(a_2z+b_2t)+g_n(\alpha,1)z+\hat{g}_{n}(\alpha,1)t+g_{2n}(\alpha,1)=0  
    \end{array} 
\end{equation*}
in $\mathrm{Spec}(\mathbb{C}[z,t,w])$. It is isomorphic to the affine curve defined by the equation
\begin{equation*}\label{eq:reducible}
    \begin{array}{ll} 
        z(a_1z+ b_1t)-\alpha & \hspace{-0.3 cm} t(a_2z+ b_2t) + (f_n(\alpha,1)-\alpha g_n(\alpha,1))z \ + \\
        &(\hat{f}_n(\alpha,1)- \alpha \hat{g}_n(\alpha,1))t + (f_{2n}(\alpha,1)- \alpha g_{2n}(\alpha,1)) = 0
    \end{array} 
\end{equation*}
in $\mathrm{Spec}(\mathbb{C}[z,t])$. There exists a constant $\alpha_1$ such that the affine curve $C_{\alpha_1} \setminus C_y$ is defined by the equation
\begin{equation*}
    (k_1z+\ell_1t+m_1)(k_2z+\ell_2t+m_2)=0,
\end{equation*}
where $k_i,\ell_i$ and $m_i$ are constants. Note that two polynomials in the equation of the affine part of ~$C_{\alpha_1}$ are not proportional, since the surface $S_{2n,2n}$ is quasi-smooth. This implies that $C_{\alpha_1}$ is defined by the quasi-homogeneous equations
\begin{equation*}
    \begin{array}{cc} 
    & (k_1z+\ell_1t+m_1y^{n})(k_2z+\ell_2t+m_2y^{n})=0,\\
    &wy+t(a_2z+b_2t)+g_n(\alpha_1,1)y^{n}z+\hat{g}_{n}(\alpha_1,1)y^{n}t+g_{2n}(\alpha_1,1)y^{2n}=0.  
    \end{array}  
\end{equation*}
Hence, the curve $C_{\alpha_1}$ consists of two irreducible reduced curves $L_{1}$ and $L'_{1}$, each of which is defined by the equations 
\begin{equation*}
    \begin{array}{cc} 
        & x-\alpha_1 y=0,\\
        & k_iz+\ell_it+m_iy^n=0,\\ 
        &wy+t(a_2z+b_2t)+g_n(\alpha_1,1)y^{n}z+\hat{g}_{n}(\alpha_1,1)y^{n}t+g_{2n}(\alpha_1,1)y^{2n}=0.\\
    \end{array} 
\end{equation*}   
\end{proof}

The following lemma provides a procedure for the successive reduction of the parameter $n$.

\begin{lemma}\label{lem:tau}
    Assume that $n\geq 2$. Let $\pi_0: Y_0\to S_{2n,2n}$ be the weighted blow-up at $\mathsf{p}_w$ with weights $(1,1)$, assigning weight one to each of the local coordinates $z$ and $t$ on the smooth cover of $\mathsf{p}_w$. Let $E$ be the exceptional divisor of $\pi_0$. Then there exists a birational morphism $\tau: Y_0\to\mathbb{P}^2$ such that $\tau(E)$ is a conic.
\end{lemma}

\begin{proof}
For a weighted blow-up $\pi_{_0} : Y_0 \to S_{2n,2n}$ at $\mathsf{p}_{_{w}}$, we have
\begin{equation*}
    -K_{Y_0} \sim_{\mathbb{Q}} \pi_{_0}^{\ast}(-K_{S_{2n,2n}})+\frac{2n-3}{2n-1} E.
\end{equation*}
    
Quasi-smoothness at $\mathsf{p}_w$ allows us to write the defining equations of $S_{2n,2n}$ as
\begin{equation*}
    wx+P_n(x,y,z,t)=wy+R_n(x,y,z,t)=0.
\end{equation*}
The binary quadratics $P_n(0,0,z,t)$ and $R_n(0,0,z,t)$ have no common projective zero. At such a zero with $x=y=w=0$, their gradients in $z,t$ would be dependent by Euler's identity, and all other first derivatives would vanish. Homogenize $P_n$ and $R_n$ to degree two in the fiber coordinates $[s:z:t]$ and set $Q_n=yP_n^h-xR_n^h$. The equation $Q_n=0$ over $\mathbb{P}^1_{[x:y]}$ gives the conic bundle on $Y_0$, since on $s\ne 0$, it is obtained by eliminating $w$. Along $E$, write $z=\xi$ and $t=\xi v$ on the blow-up of the smooth cover. Over the base chart $x\ne 0$, the fiber coordinates $s/z$ and $t/z$ correspond to
\begin{equation*}
    \frac{x^n}{z}=\xi^{2n-1}\left(\frac{x}{\xi^2}\right)^n, \ \text{and}\ \frac{t}{z}=v.
\end{equation*}
Here $x/\xi^2$ is a unit and $\xi^{2n-1}$ is a local equation of $E$ on the quotient. The other charts are analogous. At $s=0$, the two binary quadratics identify $E$ with $\mathbb{P}^1_{[z:t]}$ and give a degree-two map to the base. The discriminant has degree $2n+3$. Smoothness of $Y_0$ excludes double lines and makes the discriminant roots simple. Each singular fiber consists of two $(-1)$-curves, each meeting $E$ once. Choose a reducible member $L_1+L'_1$ over a value outside the two branch values of $E\to\mathbb{P}^1$. The strict transforms meet at a node off $E$ and meet $E$ at distinct points. $L_1$ and $L'_1$ meet on $S_{2n,2n}$ at the image of the node and at $\mathsf{p}_w$, with distinct tangent directions on the smooth cover at $\mathsf{p}_w$. For this choice of $L_1+L'_1$, we have
\begin{equation*}
    \begin{array}{cc}
        &L_1\cdot(-K_{S_{2n,2n}})=L'_1\cdot(-K_{S_{2n,2n}})=\dfrac{2}{2n-1}, \\
        &(L_1)^2=(L'_1)^2=\dfrac{2-2n}{2n-1},\ \ L_1 \cdot L'_1=\dfrac{2n}{2n-1},\ \  C_{\alpha}\cdot (-K_{S_{2n,2n}})=\dfrac{4}{2n-1},\\
    \end{array}
\end{equation*}
where $C_{\alpha}$ is a general member of the pencil $\mathscr{C}$. This implies that  
\begin{equation*}\label{eq:intersection}
    \begin{array}{cc}
    &(-K_{Y_0})^2=-(2n-5),\ \ E^2=-(2n-1),\ \ (L_1)^2=(L'_1)^2=-1,\\
    \end{array}
\end{equation*}
where we used the same notation for the strict transforms of $L_1, L'_1$ under $\pi_{_0}$. By a suitable coordinate change, we may assume that the reducible member is defined by $x=0$. Then the equation of $S_{2n,2n}$ can be written as
\begin{equation}\label{eq:coordinate}
    \begin{array}{cc} 
        &wx+zt=0, \\
        &wy+z^2+t^2+g_n(x,y)z+\hat{g}_{n}(x,y)t+g_{2n}(x,y)=0.  
    \end{array}
\end{equation}
Now we consider a hyperplane section given by $y=ax$ on $S_{2n,2n}$. Eliminating $w$ from \eqref{eq:coordinate} gives the following equation for the conic over $[1:a]$
\begin{equation*}
    -azt+z^2+t^2+g_n(1,a)x^n z+\hat{g}_{n}(1,a)x^n  t+g_{2n}(1,a)x^{2n}=0.
\end{equation*}

The symmetric matrix of this quadratic form in $(z,t,x^n)$ is
\begin{equation*}
    A(a)=
    \begin{pmatrix}
        1 & -\frac{a}{2} & \frac{g_n(1,a)}{2} \\
        -\frac{a}{2} & 1 & \frac{\hat{g}_n(1,a)}{2} \\
        \frac{g_n(1,a)}{2} & \frac{\hat{g}_n(1,a)}{2} & g_{2n}(1,a)
    \end{pmatrix}.
\end{equation*}
Choose a root $a$ of $\det A(a)$ for which $[1:a]$ lies outside the branch locus of $E\to\mathbb{P}^1$, and denote the corresponding member by $L_2+L'_2$. Replacing $y$ by $y-ax$, we may take the two selected fibers to be ~$x=0$ and $y=0$. Since $L'_1$ and $L'_2$ meet $E$ at distinct points, a linear change of $z,t$ and shifts by multiples of $y^n$ and $x^n$ put $L'_1$ and $L'_2$ in the forms $z=0$ and $t=0$ on their respective
fibers. Hence
\begin{equation*}
    Q_{n-1}(x,y,s,z',t')=\frac{Q_n(x,y,s,xz',yt')}{xy}
\end{equation*}
is a polynomial. Locally at either selected fiber, the old equation has the form $hL+uR=0$, where $u$ is a base parameter and $h=0$ is the component to be retained. The map 
\begin{equation*}
    [s:h:v] \ \mapsto \ [us:h:uv]
\end{equation*}
extends across that component by $h/u=-R/L$ away from the node and by $u/h=-L/R$ near the node, where smoothness gives $R\ne 0$. It contracts the other component to $[0:1:0]$, a smooth point of the new model since $L(0,v,s)$ is a nonzero linear form. The new fiber has equation $h'L(0,v,s)+R(0,0,v,s)=0$. Since $R$ is nonzero at the old node, this conic is smooth. Thus the substitution gives the contraction 
\begin{equation*}
    \tau_1^{(1)}: Y_0\to Y_2
\end{equation*}
of $L_1$ and $L_2$. Since every term contains $x$ or $y$, write $Q_{n-1}=yP_{n-1}^h-xR_{n-1}^h$. The corresponding weighted
complete intersection is defined by
\begin{equation*}
    w'x+P_{n-1}=w'y+R_{n-1}=0
\end{equation*}
in $\mathbb{P}(1,1,n-1,n-1,2n-3)$, both of degree $2n-2$, where $P_{n-1}$ and $R_{n-1}$ are obtained by setting $s=1$. At $s=0$, the substitution sends
\begin{equation*}
    dyz^2+(\alpha x+\beta y)zt+cxt^2 \ \mapsto \ dxz'^2+(\alpha x+\beta y)z't'+cyt'^2
\end{equation*}
for some constants $c,d,\alpha,\beta$. In both expressions the two binary quadratics have no common zero precisely when $cd(cd-\alpha\beta)\ne 0$. Hence the new surface meets $x=y=0$ only at $\mathsf{p}_{w'}$. The conic model is smooth on the weight-one charts, and the equations on the cover $w'=1$ have independent linear parts ~$x$ and $y$ at $\mathsf{p}_{w'}$. This proves quasi-smoothness. It also gives well-formedness, since the ambient space is well formed and its singular locus can meet the surface only at $\mathsf{p}_{w'}$. The induced pencil is again generated by $x,y$. The same local description identifies $Y_2$ with the weighted
blow-up of the new model at $\mathsf{p}_{w'}$, with exceptional curve $\tau_1^{(1)}(E)$. Denote this morphism by $\pi_1$. For $n=2$, it is the ordinary blow-up of the smooth point
~$\mathsf{p}_{w'}$. Then $\pi_1\circ\tau_1^{(1)}$ descends through $\pi_0$ to $\sigma_1^{(1)}$, giving the following commutative diagram:
\begin{equation*}\label{eq:link}
    \begin{tikzcd}[column sep=3em, row sep=3em]
            Y_0 \arrow[dd, "\pi_{_0}"'] \arrow[rr, "\tau_{1}^{(1)}"] & & Y_2 \arrow[dd, "\pi_{_1}"] \\
            & & \\
            S_{2n,2n} \arrow[rr, "\sigma_{1}^{(1)}"] & & S_{2n-2,2n-2}
    \end{tikzcd}
\end{equation*}
For $n=2$, the first contraction already gives $S_{2,2}$. For $n\geq 3$, applying the same construction to $\pi_1$ gives $\tau_1^{(2)}$, $\sigma_1^{(2)}$, and $\pi_2$. Repeating the step
gives the birational morphisms
\begin{equation*}
    \tau_1: Y_0\to Y_{2n-2},\ \sigma_1: S_{2n,2n}\to S_{2,2},
\end{equation*}
where $\tau_1=\tau_1^{(n-1)}\circ\cdots\circ\tau_1^{(1)},\ \sigma_1=\sigma_1^{(n-1)}\circ\cdots\circ\sigma_1^{(1)}$.
\begin{equation*}\label{eq:links}
    \begin{tikzcd}[column sep=2.2em, row sep=2.2em]
            Y_0 \arrow[dd, "\pi_{_0}"] \arrow[rr, "\tau_{1}^{(1)}"] & & Y_2 \arrow[dd, "\pi_{_1}"] \arrow[rr, "\tau_{1}^{(2)}"] & & Y_4 \arrow[r] \arrow[dd, "\pi_{_2}"] & \cdots \arrow[r] & Y_{2n-4} \arrow[dd, "\pi_{_{n-2}}"] \arrow[rr, "\tau_{1}^{(n-1)}"] & & Y_{2n-2} \arrow[dd, "\pi_{_{n-1}}"] \\
            & & & & & \cdots & & & \\
            S_{2n,2n} \arrow[rr, "\sigma_{1}^{(1)}"] & & S_{2n-2,2n-2} \arrow[rr, "\sigma_{1}^{(2)}"] & & S_{2n-4,2n-4} \arrow[r] & \cdots \arrow[r] & S_{4,4} \arrow[rr, "\sigma_{1}^{(n-1)}"] & & S_{2,2}
    \end{tikzcd}
\end{equation*}
Each step makes its two selected fibers smooth, so the curves contracted in successive steps lie in distinct fibers of the original conic bundle. Let $L_1,\ldots,L_{2n-2}$ be the disjoint $(-1)$-curves contracted by $\tau_1$. Put ~$\overline E=\tau_1(E)$. Since each $L_i$ meets $E$ once,
\begin{equation*}
    K_{Y_{2n-2}}^2=(5-2n)+(2n-2)=3,\ \overline E^2=-1, \ (-K_{Y_{2n-2}})\cdot\overline E=1.
\end{equation*}
For any irreducible curve $C\ne\overline E$ on $Y_{2n-2}$, with strict transform $\widetilde C$ on $Y_0$, we have
\begin{equation*}
    (-K_{Y_{2n-2}})\cdot C=(-K_{Y_0})\cdot\widetilde C+\sum_{i=1}^{2n-2}L_i\cdot\widetilde C>0.
\end{equation*}
The first term is positive by the discrepancy formula for $\pi_0$, and the remaining terms are nonnegative. Then the Nakai--Moishezon criterion implies the ampleness of $-K_{Y_{2n-2}}$. Hence $Y_{2n-2}$ is a smooth cubic surface in its anticanonical embedding, and $\overline E$ is a line. On $Y_0$, the fiber class satisfies 
\begin{equation*}
    F\sim_{\mathbb Q}\pi_0^{\ast}(-K_{S_{2n,2n}})-\frac{2}{2n-1}E,
\end{equation*}
since $x$ and $y$ have order two on the smooth cover. This implies that $F\sim_{\mathbb Q}-K_{Y_{2n-2}}-\overline E$ on $Y_{2n-2}$. Choose a line $L_{2n+4}$ disjoint from $\overline E$. Since $F\cdot L_{2n+4}=1$, this line is a section of the induced conic bundle. In each of the five remaining singular fibers,
choose the component disjoint from $L_{2n+4}$, and denote these components by $L_{2n-1},\ldots,L_{2n+3}$.
These six disjoint $(-1)$-curves contract to $\mathbb{P}^2$ under a morphism $\tau_2$. The five fiber components each meet $\overline E$ once, so $\tau_2(\overline E)$ is a conic. The six curves avoid the blow-up centers of $\tau_1$, which lie on $\overline E$ over the fibers smoothed by $\tau_1$.
We use the same notation for their strict transforms on $Y_0$ and set $\tau=\tau_2\circ\tau_1$.
\end{proof}

\begin{lemma}\label{lem:pi}
    There exists a birational morphism $\pi : Y_0 \to \mathbb{P}(1,1,2n-1)$.
\end{lemma}

\begin{proof}
Let $\mathsf{p}_{_1},\dots,\mathsf{p}_{_{2n+4}}$ be the points in $\mathbb{P}^2$ that are the images of the $(-1)$-curves $L_{1},\dots, L_{2n+4}$ of $\tau$ in Lemma \ref{lem:tau}. Then $\tau(E)$ is a conic that passes through the points $\mathsf{p}_{_1},\dots,\mathsf{p}_{_{2n+3}}$ but not ~$\mathsf{p}_{_{2n+4}}$. Let $M_i$ be the line on $\mathbb{P}^2$ passing through the points $\mathsf{p}_{_i}$ and $\mathsf{p}_{_{2n+3}}$, and let $\widetilde{M}_i$ be the strict transform of $M_i$ under ~$\tau$, where $i=1,\dots,2n+2$. The point $\mathsf{p}_{2n+4}$ lies on none of the lines $M_i$. Indeed, a line through $\mathsf{p}_{2n+4}$ and two of the points on the conic would have a strict transform of anticanonical degree zero on $Y_0$, contrary to $(-K_{Y_0})\cdot C>0$ for every irreducible curve $C\ne E$. Then 
\begin{equation*}
    \widetilde{M}_1,\dots,\widetilde{M}_{2n+2}\ \mathrm{and} \ \ L_{2n+4}
\end{equation*}
are disjoint $(-1)$-curves on $Y_0$, none of which intersect $E$. The pencil of lines through $\mathsf{p}_{2n+3}$ induces a ruling on $Y_0$. Contracting the above curves makes all its fibers
smooth. The image of $E$ is a section of self-intersection $-(2n-1)$, so the resulting ruled surface is $\mathbb{F}_{2n-1}$. Now let 
\begin{equation*}
    \eta_{_1} : Y_0 \to \mathbb{F}_{2n-1} \ (\mathrm{resp.} \ \eta_{_2} : \mathbb{F}_{2n-1} \to \mathbb{P}(1,1,2n-1))
\end{equation*}
be the birational morphism obtained by contracting the $2n+3$ disjoint $(-1)$-curves, 
\begin{equation*}
    \widetilde{M}_1, \dots, \widetilde{M}_{2n+2} \ \ \mathrm{and}\ \ L_{2n+4} \ (\mathrm{resp.\ the\ negative\ section\ of\ } \mathbb{F}_{2n-1}).
\end{equation*}

The proof is completed by letting 
\begin{equation*}
    \pi=\eta_2\circ\eta_1=\eta_0\circ\pi_0
\end{equation*}
where $\eta_{_0} : S_{2n,2n} \to \mathbb{P}(1,1,2n-1)$ is the induced map.
\end{proof}

Figure~\ref{fig:total} summarizes the birational maps described above.

\begin{figure}[H]
\centering
    \begin{tikzcd}[column sep=1.8em, row sep=2.5em]
    & & & & & & Y_0
        \arrow[dllllll, "\eta_{1}"']
        \arrow[dll, "\pi_{0}"]
        \arrow[drr, "\tau_{1}"] & & & & \\
    \mathbb{F}_{2n-1} \arrow[dd, "\eta_{2}"']
        & & & & S_{2n,2n}
        \arrow[ddllll, "\eta_{0}"]
        \arrow[drr, "\sigma_{1}"]
        & & & & Y_{2n-2}
        \arrow[drr, "\tau_{2}"]
        \arrow[dll, "\pi_{n-1}"] & & \\
    & & & & & & S_{2,2} \arrow[drr, "\sigma_{2}"]
        & & & & \mathbb{P}^2 \arrow[dll,dashed] \\
    \mathbb{P}(1,1,2n-1) & & & & & & & & \mathbb{P}^2 & &
    \end{tikzcd}
\caption{Birational models of $S_{2n,2n}$.}
\label{fig:total}
\end{figure}

As previously noted, let $D\sim_{\mathbb{Q}} -K_{S_{2n,2n}}$ be an effective $\mathbb{Q}$-divisor.

\begin{lemma}\label{lem:pw}
    The log pair $(S_{2n,2n},D)$ is log canonical at every point of $S_{2n,2n} \setminus \left \{ \mathsf{p_{_w}} \right \}$, if $n\geq 2$.
\end{lemma}

\begin{proof}
We divide the proof into the cases $n=2$ and $n \geq 3$, which are treated in Cases 1 and 2, respectively. Let $\mathsf{p}$ be a smooth point in $S_{2n,2n} \setminus \left \{ \mathsf{p_{_w}} \right \}$ with $n \geq 2$.\\

\textbf{Case 1.} Suppose that the log pair $(S_{4,4},D)$ is not log canonical at $\mathsf{p}$ in $\mathbb{P}(1,1,2,2,3)$.\\ 
We identify $S_{4,4}\setminus\{\mathsf{p}_w\}$ with $Y_0\setminus E$ via $\pi_0$. The pencil generated by $x$ and $y$ on $S_{4,4}$ lifts via $\pi_0$ to a conic bundle on $Y_0$ with fiber class $F$. Then we have 
\begin{equation*}
    -K_{Y_0}\sim F+E, \qquad K_{Y_0}^2=1,
\end{equation*}
and $E^2=-3$. This implies that $(-K_{Y_0})\cdot C>0$ for every irreducible curve $C\ne E$. Suppose first that ~$\mathsf{p}$ lies on a reducible member $L_1+L'_1$. The pair $(S_{4,4},L_1+L'_1)$ is log canonical at ~$\mathsf{p}$, since the strict transforms meet transversely. By Lemma~\ref{lem:convex},
we may assume that 
\begin{equation*}
    L'_1\not\subset\supp(D).
\end{equation*}
If $\mathsf{p}\in L'_1$, then by Lemma~\ref{lem:multiplicity},
\begin{equation*}
    \frac{2}{3}=D\cdot L'_1\geq\mult_{\mathsf{p}}(D)>1,
\end{equation*}
a contradiction. Hence $\mathsf{p}\in L_1\setminus L'_1$ and write 
\begin{equation*}
    D=aL_1+\Delta,
\end{equation*}
where $L_1\not\subset\supp(\Delta)$. Since $\frac23=D\cdot L'_1\geq\frac{4a}{3}$, we have $a\leq\frac{1}{2}$. By applying Lemma~\ref{lem:adjuction} with $r=1$ to the non-log-canonical pair $(S_{4,4},L_1+\Delta)$, we also have a contradiction,
\begin{equation*}
    1<L_1\cdot\Delta=\frac{2+2a}{3}\leq 1.
\end{equation*}
Now suppose that $\mathsf{p}$ lies on a smooth fiber $F_c$. Choose a reducible fiber $F_b$ away from the branch values of $E\to \mathbb{P}^1$, and choose base coordinates in which $b$ and $c$ are finite. On $x=1$, the components of $F_b$ are given by 
\begin{equation*}
    y=b, \ \ell_i=\alpha_i z+\beta_i t+\gamma_i=0
\end{equation*}
for $i=1,2$. Their distinct intersections with $E$ imply that the linear parts of $\ell_1$ and $\ell_2$ are independent. Let $Z_i$ be obtained by contracting the component $\ell_j=0$, where $j\ne i$. Then $K_{Z_i}^2=2$, and the image of ~$E$ is a $(-2)$-curve. The pullback of $-K_{Z_i}$ is linearly equivalent to $F+E$ plus the contracted component. Thus $-K_{Z_i}$ has degree zero on the image of $E$ and positive degree on every other irreducible curve. Its anticanonical model $X_i$ is a degree-two del Pezzo surface with a single $A_1$ singularity.\\
After the removal of the common exceptional contribution $\frac{1}{3} E$ and the retained component $\ell_i=0$, the weighted homogeneous polynomials of degree two
\begin{equation*}
    x(y-bx),\ y^2-b^2x^2, \ \text{and}\ \alpha_i z+\beta_i t+\gamma_i x^2
\end{equation*}
give a basis of $\mathrm{H}^0(Z_i,-K_{Z_i})$. Near $\mathsf{p}$, the resulting
map is given by
\begin{equation*}
    \left[1:y+b:\frac{\ell_i}{y-b}\right].
\end{equation*}
Choose $i$ so that the differential $\mathrm{d}\ell_i$ is nonzero on the tangent line to $F_c$ at $\mathsf{p}$. On this line, $\mathrm{d}y$ vanishes and $\mathrm{d}\bigl(\ell_i/(y-b)\bigr)$ equals $(c-b)^{-1}\mathrm{d}\ell_i$. Since $F_c$ is smooth, $\mathrm{d}y$ is nonzero at $\mathsf{p}$. This implies that the two differentials are linearly independent there, and the image of $\mathsf{p}$ on $X_i$ lies outside the ramification divisor. Write 
\begin{equation*}
    \pi_0^\ast (D)=\widetilde D+\mu E.
\end{equation*}
The effective log pullback boundary $\widetilde D+(\mu+\frac{1}{3})E\sim_{\mathbb{Q}}-K_{Y_0}$
pushes forward to an anticanonical divisor on $X_i$. Since the contractions are isomorphisms near $\mathsf{p}$, the resulting pair on $X_i$ fails to be log canonical at the image of $\mathsf{p}$, which contradicts \cite[Lemma~3.2]{Cheltsov2016}.\\

\textbf{Case 2.} Suppose that the log pair $(S_{2n,2n},D)$ is not log canonical at $\mathsf{p}$ with $n\geq 3$.\\ 
Let $C\in\left\vert\mathscr{O}_{S_{2n,2n}}(1)\right\vert$ be the member passing through $\mathsf{p}$. On the chart $x=1$ or $y=1$ containing ~$\mathsf{p}$, eliminating $w$ expresses the affine part of $C$ as a plane conic. Then the quasi-smoothness excludes double-line fibers in the conic bundle described in the proof of Lemma~\ref{lem:tau}. This implies that $C$ is smooth or nodal at $\mathsf{p}$, and the pair $(S_{2n,2n},C)$ is log canonical at the same point. If $C$ is irreducible, then 
\begin{equation*}
    C\sim_{\mathbb{Q}}D,
\end{equation*}
and Lemma~\ref{lem:convex} allows us to assume that $C\not\subset\supp(D)$. By Lemma~\ref{lem:multiplicity}, we have a contradiction
\begin{equation*}
    1>\frac{4}{2n-1}=C\cdot D
    \geq\mult_{\mathsf{p}}(C)\cdot\mult_{\mathsf{p}}(D)>1.
\end{equation*}
If $C$ is reducible, write $C=L_1+L'_1$. Again by Lemma~\ref{lem:convex}, we may assume that
\begin{equation*}
    L'_1\not\subset\supp(D).
\end{equation*}
If $\mathsf{p}\in L'_1$, then
\begin{equation*}
    1>\frac{2}{2n-1}=L'_1\cdot D \geq\mult_{\mathsf{p}}(D)>1,
\end{equation*}
a contradiction. Hence, $\mathsf{p}\in L_1\setminus L'_1$. If $L_1\not\subset\supp(D)$,
then we also have a contradiction
\begin{equation*}
    1>\frac{2}{2n-1}=L_1\cdot D
    \geq\mult_{\mathsf{p}}(L_1)\cdot\mult_{\mathsf{p}}(D)>1.
\end{equation*}
This implies that $L_1\subset\supp(D)$, and we write 
\begin{equation*}
    D=aL_1+\Delta,
\end{equation*}
where $a>0$ and $L_1\not\subset\supp(\Delta)$. Since
\begin{equation*}
    L'_1\cdot D=aL'_1\cdot L_1+L'_1\cdot\Delta\geq aL'_1\cdot L_1,
\end{equation*}
we have
\begin{equation*}
    1>\frac{1}{n}=\frac{L'_1\cdot D}{L'_1\cdot L_1}\geq a.
\end{equation*}
Applying Lemma~\ref{lem:adjuction} with $r=1$ to the non-log-canonical pair $(S_{2n,2n},L_1+\Delta)$ gives
\begin{equation*}
    \frac{2+a(2n-2)}{2n-1}=L_1\cdot(D-aL_1)=L_1\cdot\Delta \geq\mult_{\mathsf{p}}(\Delta\vert_{L_1})>1.
\end{equation*}
This yields a contradiction,
\begin{equation*}
    \frac{1}{n}\geq a>\frac{2n-3}{2n-2}.
\end{equation*}
\end{proof}   

\begin{proposition}\label{prop:support2} 
    If a log pair $(S_{2n,2n},D)$ is not log canonical at a point $\mathsf{p}$ with $n \geq 2$, then there exists an effective anticanonical divisor $T$ such that $(S_{2n,2n},T)$ is not log canonical at $\mathsf{p}$ and the support of $T$ is contained in the support of $D$.
\end{proposition}

\begin{proof}
Let $(S_{2n,2n},D)$ be a log pair with $n \geq 2$, which is not log canonical at a point $\mathsf{p}$. Then by Lemma ~\ref{lem:pw}, $\mathsf{p}$ has to be $\mathsf{p}_{_{w}}$. For the weighted blow-up $\pi_{_0} : Y_0 \to S_{2n,2n}$ at $\mathsf{p}_{_{w}}$ in Lemma \ref{lem:tau}, let $\widetilde{D}$ be the strict transform of $D$ of $\pi_{_0}$. Then we have an effective $\mathbb{Q}$-divisor
\begin{equation*}
    \widetilde{D}+\frac{\alpha}{2n-1}E \sim_{\mathbb{Q}} -K_{Y_0}
\end{equation*}
with some positive rational number $\alpha$ such that the log pair $\left (Y_0,\widetilde{D}+\frac{\alpha}{2n-1}E \right )$ is not log canonical at a point $\mathsf{q}$ on $E$. Furthermore, for the birational morphism $\tau_{_1}$ in Figure~\ref{fig:total}, let 
\begin{equation*}
    \overline{D}+\frac{\alpha}{2n-1}\overline{E} \sim_{\mathbb{Q}} -K_{Y_{2n-2}}
\end{equation*}
be its image on $Y_{2n-2}$, a smooth del Pezzo surface  of degree $3$. Then the log pair
\begin{equation*}
    \left (Y_{2n-2}, \overline{D}+\frac{\alpha}{2n-1}\overline{E} \right )
\end{equation*}
is also not log canonical at $\tau_{_1}(\mathsf{q})$. By \cite[Theorem~1.12]{Cheltsov2016}, this implies that there exists a unique effective anticanonical divisor $\overline{T}$ such that $(Y_{2n-2},\overline{T})$ is not log canonical at $\tau_{_1}(\mathsf{q})$ and the support of $\overline{T}$ is contained in the support of $\overline{D}+\frac{\alpha}{2n-1}\overline{E}$. Since $\overline{E}$ is a line on $Y_{2n-2}$, $\overline{T}$ has to be one of the following:
\begin{itemize}
    \item[$\bullet$] $F_1+F_2+\overline{E}$, three concurrent lines meeting pairwise transversely at $\tau_{_1}(\mathsf{q})$ or
    \item[$\bullet$] $C+\overline{E}$, a conic and a line tangent at $\tau_{_1}(\mathsf{q})$.
\end{itemize}

Meanwhile, if there exists $L_i$ on $Y_0$ passing through the point $\mathsf{q}$ among the $(-1)$-curves contracted by $\tau_{_1}$ such that $\supp(L_i) \not\subset \supp(\widetilde{D})$, then it gives a contradiction,
\begin{equation*}
    1 = L_i \cdot (-K_{Y_0}) = L_i \cdot \left (\widetilde{D}+\frac{\alpha}{2n-1}E \right ) \geq \mult_{\mathsf{q}}(L_i)\cdot \mult_{\mathsf{q}} \left (\widetilde{D}+\frac{\alpha}{2n-1}E \right ) > 1.
\end{equation*}
Hence, for the strict transform $\widetilde{T}$ of $\overline{T}$ under $\tau_{_1}$, one of the following holds:
\begin{equation*}\label{eq:T-tilde}
    \left\{
    \begin{array}{ll}
    \widetilde{T} \sim -K_{Y_0},\ \mathrm{if \ there \ exists \ no} \ L_i \ \mathrm{passing \ through  \ \mathsf{q}}, \\  
    \widetilde{T}+2L_i \sim -K_{Y_0},\ \mathrm{if} \ \overline{T}=F_1+F_2+\overline{E} \ \mathrm{and \ there \ exists} \ L_i \ \mathrm{passing \ through \ \mathsf{q}}, \\
    \widetilde{T}+L_i \sim -K_{Y_0},\ \mathrm{if} \  \overline{T}=C+\overline{E} \ \mathrm{and \ there \ exists} \ L_i \ \mathrm{passing \ through \ \mathsf{q}}.\\
    \end{array} \right.
\end{equation*}
In each case, the divisor on the left-hand side is effective and its support is contained in \(\operatorname{Supp}\bigl(\widetilde D+\frac{\alpha}{2n-1}E\bigr)\). Its pushforward under \(\pi_0\) is the required divisor \(T\).
\end{proof}

Now we reach the stage to state our result for existence.

\begin{theorem}\label{thm:cyl}
    The complete intersection $S_{2n,2n}$ in $\mathbb{P}(1,1,n,n,2n-1)$ admits an anticanonical polar cylinder if and only if $n=1$.
\end{theorem}

\begin{proof}
When $n=1$, the surface is a smooth del Pezzo surface of degree $4$, for which the existence of an anticanonical polar cylinder was proved in \cite{Kishimoto2009}. Suppose that $S_{2n,2n}$ contains an anticanonical polar cylinder with $n\geq 2$, and let $D\sim_{\mathbb{Q}}-K_{S_{2n,2n}}$ be an effective $\mathbb{Q}$-divisor such that 
\begin{equation*}
    U=S_{2n,2n}\setminus\supp(D)\cong\mathbb{A}^1\times Z
\end{equation*}
for some affine curve $Z$. Let $\mathscr{L}$ be the pencil induced by the cylinder. If $\mathscr{L}$ is base-point-free, the unique horizontal component $C_1\subset\supp(D)$ is a section. For a general member $F$ of the pencil, we have 
\begin{equation*}
    0=(K_{S_{2n,2n}}+D)\cdot F=-2+\lambda_1,
\end{equation*}
where $\lambda_1$ is the coefficient of $C_1$ in $D$. This implies that the log pair $(S_{2n,2n},D)$ is not log canonical at every point of $C_1$, contrary to Lemma~\ref{lem:pw}. Hence, the pencil $\mathscr{L}$ has the base point $\mathsf{p}$, and Lemmas~\ref{lem:bdry} and~\ref{lem:pw} imply $\mathsf{p}=\mathsf{p}_w$.\\
Let $\psi: W\to S_{2n,2n}$ be a resolution of indeterminacy of the rational map defined by $\mathscr{L}$, chosen to factor through $\pi_0$ in Lemma~\ref{lem:tau} and to be an isomorphism over $U$. Write $\widetilde{\rho}: W\to\mathbb{P}^1$ for the induced morphism. A general fiber meets $W\setminus U$ in one point, and its unique horizontal boundary component $G$ is a section, which is exceptional over $\mathsf{p}_w$. For any effective 
\begin{equation*}
    H\sim_{\mathbb{Q}}-K_{S_{2n,2n}}
\end{equation*}
with $\supp(H)\subset\supp(D)$, let $b_G$ be the coefficient of $G$ in the log pullback boundary
\begin{equation*}
    \psi^{\ast}(K_{S_{2n,2n}}+H)-K_W.
\end{equation*}
Intersecting with a general fiber implies $b_G=2$, and $G$ has log discrepancy $-1$ with respect to the pair ~$(S_{2n,2n},H)$. For $\pi_0^{\ast} (H)=\widetilde H+d_HE$, the log pullback boundary on $Y_0$ is
\begin{equation*}
    B_H=\widetilde H+c_HE\sim_{\mathbb{Q}}-K_{Y_0},
\end{equation*}
where $c_H=d_H+\frac{2n-3}{2n-1}>0$. For the fiber class $F_0$ of the conic bundle on $Y_0$ induced by the pencil generated by $x$ and $y$ on $S_{2n,2n}$, we have
\begin{equation*}
    2=B_H\cdot F_0=\widetilde H\cdot F_0+2c_H,
\end{equation*}
which implies $c_H\leq 1$. Since $c_H < b_G=2$, $G$ is contracted to a point $\mathsf{q}$ of $E$, and the log pair $(Y_0,B_H)$ is not log canonical at $\mathsf{q}$ for every such $H$. The conic bundle on $Y_0$ restricts to a double cover $E$ of $\mathbb{P}^1$ with two branch values. Choose $\tau_1: Y_0\to Y_{2n-2}$ as in Figure~\ref{fig:total}, contracting one component from each of ~$2n-2$ distinct singular fibers that avoid the point $\mathsf{q}$ and lie over nonbranch values of this cover. These conditions exclude at most three of the $2n+3$ singular fibers. Since $\tau_1$ is an isomorphism near $\mathsf{q}$, this implies that a log pair 
\begin{equation*}
    (Y_{2n-2},\tau_{1\ast}(B_H)) 
\end{equation*}
is not log canonical at $\tau_1(\mathsf{q})$. Let $\overline T$ be the tangent hyperplane section at $\tau_1(\mathsf{q})$. By \cite[Corollary~1.13]{Cheltsov2016}, the log pair $(Y_{2n-2},\overline T)$ is not log canonical at $\tau_1(\mathsf{q})$ and
\begin{equation}\label{eq:contradiction}
    \supp(\overline T)\subset\supp(\tau_{1\ast}(B_H))
\end{equation}
for every such $H$. The section $\overline T$ contains the line $\overline E=\tau_1(E)$, and its residual curve is either
\begin{itemize}
    \item[$\bullet$] two lines through $\tau_1(\mathsf{q})$ or
    \item[$\bullet$] a conic tangent to $\overline E$ at $\tau_1(\mathsf{q})$.
\end{itemize}
Hence, the residual curve avoids the blow-up centers of $\tau_1$, which lie on $\overline E$ away from $\tau_1(\mathsf{q})$. Thus $\overline T$ has multiplicity one at each center, and its strict transform $\widetilde T$ is anticanonical on $Y_0$. Taking $H=D$, we obtain an effective anticanonical divisor $T=\pi_{0\ast}(\widetilde T)$  with at most two irreducible components and 
\begin{equation*}
    \supp(T)\subset\supp(D).
\end{equation*}
Since $Z$ is a smooth rational affine curve, $\operatorname{Cl}(U)=0$, and the components of $D$ generate $\operatorname{Cl}(S_{2n,2n})_{\mathbb{Q}}$, whose rank is $2n+4$. Therefore, $D\ne T$ and by Lemma~\ref{lem:convex}, there exists an effective $\mathbb{Q}$-divisor
\begin{equation*}
    D'\sim_{\mathbb{Q}}-K_{S_{2n,2n}}
\end{equation*}
with $\supp(D')\subset\supp(D)$ and $\supp(T)\not\subset\supp(D')$. The omitted component of $T$ corresponds to a component of $\overline T$ absent from $\supp(\tau_{1\ast}(B_{D'}))$, which contradicts \eqref{eq:contradiction} for $H=D'$.
\end{proof}

Theorems \ref{thm:absence} and \ref{thm:cyl} along with Corollary \ref{cor:absence} imply the Main Theorem in Section \ref{sec:intro}.\\ 

As a final remark, we note a result in \cite{Sawahara2025}, which establishes the ample polar cylindricity of the surfaces obtained by blow-ups of $\mathbb{P}(1,1,2n-1)$ along $\eta_{_0}$ in Figure~\ref{fig:total}.

\begin{remark}\label{rmk:Sawahara}
    In \cite[Theorem 4.1]{Sawahara2025}, it is proved that a normal rational surface $S$ contains an $H$-polar cylinder for any ample divisor $H$ on $S$ if the following conditions hold for some non-negative integer $k$:  
    \begin{enumerate}
    \item[(1)] A minimal resolution $f: \widetilde{S} \to S$ has a $\mathbb{P}^1$-fibration $g: \widetilde{S} \to \mathbb{P}^1$ such that 
        \begin{itemize}
        \item[$\bullet$] $g$ admits a section with self-intersection number $-k$;
        \item[$\bullet$] all singular fibers of $g$ consist only of $(-1)$-curves and $(-2)$-curves;
        \item[$\bullet$] every irreducible curve on $\widetilde{S}$ with self-intersection at most $-2$ is contained in a singular fiber of $g$, with the possible exception of the section.
        \end{itemize}

    \item[(2)] All irreducible curves on $\widetilde{S}$ with self-intersection number at most $-2$ are contracted by $f$ with
        \begin{equation*}
            (-K_{\widetilde{S}})^2 \geq 5 - k.
        \end{equation*}
    \end{enumerate}
For $n\geq 2$, put $k=2n-1$. By Lemma~\ref{lem:pi}, the minimal resolution $Y$ of each intermediate surface is obtained by blowing up $m\leq 2n+2$ points on distinct fibers of $\mathbb{F}_k$ away from the negative section. The section retains self-intersection $-k$, and each singular fiber consists of two $(-1)$-curves. The discrepancy formula and adjunction give $C^2\geq -1$ for every irreducible curve $C\ne E$ on $Y_0$. Thus the image of $E$ is the only curve on $Y$ with self-intersection at most $-2$, and the resolution map contracts it. Since ~$(-K_Y)^2=8-m\geq 5-k$, each intermediate surface contains an
$H$-polar cylinder for every ample ~$\mathbb{Q}$-divisor $H$ by \cite[Theorem~4.1]{Sawahara2025}. For $S_{2n,2n}$ itself, Theorem~\ref{thm:cyl} excludes polar cylinders even for the anticanonical polarization.
\end{remark}


\newpage

\section{Appendix}\label{sec:appendix}
The tables below list the quasi-smooth well-formed complete intersection log del Pezzo surfaces of index one. An asterisk marks each principal case, and $n$ ranges over the positive integers.

\begin{table}[H]
    \begin{threeparttable}
    \setlength{\tabcolsep}{7pt}
    \centering
    \caption{Quasi-smooth Well-formed Log del Pezzo Hypersurfaces of Index One}
  \label{table_hypersurface}
    \begin{tabular}{|c c c||c c c|} 
     \hline 
     \emph{No.} & $(a_0,a_1,a_2,a_3,d)$ & lct & \emph{No.} & $(a_0,a_1,a_2,a_3,d)$ & lct\\ [0.5ex] 
     \hline\hline 

     \emph{1} &   (1, 1, 1, 1, 3) & $\begin{array}{rl}
        \frac{3}{4}^a \\ 
        \frac{2}{3}^b
        \end{array}$ 
     & \emph{13} &  (5, 14, 17, 21, 56)& $\frac{25}{8}$ \\[3.5ex]

     \emph{2} & (1, 1, 1, 2, 4) & $\begin{array}{rl}
        \frac{5}{6} ^c \\ 
        \frac{3}{4} ^d
        \end{array}$ 
     & \emph{14}& (5, 19, 27, 31, 81) & $\frac{25}{6}$ \\[3.5ex]

     \emph{3} &  (1, 1, 2, 3, 6) & $\begin{array}{rl}
        1 ^e \\ 
        \frac{5}{6}^f
        \end{array}$ 
     & \emph{15}& (5, 19, 27, 50, 100) & $\frac{25}{6}$ \\[3.5ex]

     \emph{4*} &  (1, 2, 3, 5, 10) & 
     $\begin{array}{rl}
        1 ^g \\ 
        \frac{7}{10}^h
        \end{array}$ 
     & \emph{16} &   (7, 11, 27, 37, 81) & $\frac{49}{12}$ \\[3.5ex]

     \emph{5*} & (1, 3, 5, 7, 15) & $\begin{array}{rl}
        1^i \\ 
        \frac{8}{15}^j
        \end{array}$ 
     & \emph{17} & (7, 11, 27, 44, 88) & $\frac{35}{8}$\\[3.5ex]

     \emph{6} & (1, 3, 5, 8, 16)  & $1$
     & \emph{18}& (9, 15, 17, 20, 60) & $\frac{21}{4}$ \\[3.5ex]

     \emph{7} & (2, 3, 5, 9, 18) & $\begin{array}{rl}
        2 ^k \\ 
        \frac{11}{6}^l
        \end{array}$ 
     & \emph{19} & (9, 15, 23, 23, 69) & $6$ \\[3.5ex]

     \emph{8} & (2, 2n+1, 2n+1, 4n+1, 8n+4) & $1$
     & \emph{20} & (11, 29, 39, 49, 127) & $\frac{33}{4}$\\[3ex]

     \emph{9} & (3, 3, 5, 5, 15) & $2$
     & \emph{21}& (11, 49, 69, 128, 256) & $\frac{55}{6}$\\[3ex]

     \emph{10} &  (3, 5, 7, 11, 25)& $\frac{21}{10}$
     & \emph{22} &  (13, 23, 35, 57, 127)& $\frac{65}{8}$ \\[3ex]

     \emph{11} &  (3, 5, 7, 14, 28) & $\frac{9}{4}$
     & \emph{23} & (13, 35, 81, 128, 256) & $\frac{91}{10}$ \\[3ex]

     \emph{12}& (3, 5, 11, 18, 36) & $\frac{21}{10}$
     &  & &\\ [1ex] 
     \hline
\end{tabular}  

\begin{tablenotes}\footnotesize
    
    \item[$a$] $S_3$ has no Eckardt points 
    \item[$b$] $S_3$ has an Eckardt point 
    \item[$c$] $\vert -K_{S_4} \vert $ contains no tacnodal curves
    \item[$d$] $\vert -K_{S_4} \vert $ contains a tacnodal curve
    \item[$e$] $\vert - K_{S_6} \vert$ contains no cuspidal curves
    \item[$f$] $\vert - K_{S_6} \vert$ contains a cuspidal curve
    \item[$g$] $H_x$ has an ordinary double point
    \item[$h$] $H_x$ has a non-ordinary double point 
    \item[$i$] the defining equation of $S_{15}$ contains $yzt$: $xf_{14}+y^5+cyzt+z^3=0$, $c\ne 0$
    \item[$j$] the defining equation of $S_{15}$ contains no $yzt$: $c=0$ in the equation
    \item[$k$] $H_y$ has a tacnodal point
    \item[$l$] $H_y$ has no tacnodal points
    \end{tablenotes} 
\end{threeparttable}
\end{table}

\newpage

\begin{table}[H]
    \begin{threeparttable}
    \setlength{\tabcolsep}{5pt}
    \centering
    \caption{Complete Intersection Log del Pezzo Surfaces of Index One}
    \label{table_complete intersection}
    \begin{tabular}{|c c c||c c c|} 
     \hline 
     \emph{No.} & $(a_0,a_1,a_2,a_3,a_4,d_1,d_2)$ & lct & \emph{No.} & $(a_0,a_1,a_2,a_3,a_4,d_1,d_2)$ & lct\\ [0.5ex] 
     \hline\hline 

     \emph{1*}& (1,1,n,n,2n-1,2n,2n)& $\leq\frac{2n+3}{4n+2}$
     & \emph{21} & (4,5,7,10,13,18,20)  & $\geq 2$\\ [1.5ex]

     \emph{2}& (1,2,2,3,3,4,6) & $1$
     & \emph{22} &  (5,7,10,14,23,28,30) & $\frac{35}{12}$\\[1.5ex]

     \emph{3*} & (1,2,3,4,5,6,8)& $\begin{array}{rl}
        1 ^a \\ 
        \frac{7}{12}^b
        \end{array}$
     & \emph{23}& (5,9,12,20,31,36,40) & $\frac{55}{24}$\\ [1.5ex]

     \emph{4} & (1,2,2n+1,2n+1,4n+1,4n+2,4n+3)& $1$
     & \emph{24} &  (5,14,17,21,37,42,51)& $\frac{10}{3}$ \\[1.5ex]

     \emph{5} &(1,3,3,5,5,6,10) & $1$
     & \emph{25} &  (6,7,9,11,14,18,28) & $\geq\frac{7}{2}$ \\[1.5ex]

     \emph{6}& (1,4,5,7,11,12,15) & $1$
     & \emph{26} & (6,8,9,11,13,22,24) & $\geq 3$\\[1.5ex]

     \emph{7}& (1,4,7,10,13,14,20) & $1$
     & \emph{27} &  (9,15,23,23,31,46,54)& $\geq\frac{23}{6}$\\[1.5ex]

     \emph{8} & (1,5,8,12,19,20,24)& $1$
     & \emph{28} &  (9,15,23,23,37,46,60) & $\frac{45}{8}$ \\[1.5ex]

     \emph{9}& (1,5,9,13,17,18,26) & $1$
     & \emph{29} & (9,23,30,38,67,76,90) & $\frac{81}{14}$\\[1.5ex]

     \emph{10} & (1,7,11,17,27,28,34) & $1$
     & \emph{30} &   (10,17,25,34,43,60,68) & $6$\\ [1.5ex]

     \emph{11}& (1,7,12,17,23,24,35)& $1$
     & \emph{31} &  (11,18,27,44,61,72,88) & $\frac{77}{16}$\\ [1.5ex]

     \emph{12}& (1,8,13,19,31,32,39) & $1$
     & \emph{32} & (11,27,36,62,97,108,124) & $\frac{121}{24}$ \\[1.5ex]

     \emph{13}& (1,9,15,23,23,24,46)& $1$
     & \emph{33} & (11,29,38,48,85,96,114) & $\frac{99}{14}$ \\[1.5ex]

     \emph{14} & (2,2,3,3,3,6,6) & $\geq\frac{6}{5}$
     & \emph{34} & (11,29,39,49,59,88,98) & $\geq\frac{117}{16}$\\[1.5ex]

     \emph{15} &  (2,3,4,5,5,8,10) & $\geq\frac{9}{8}$
     & \emph{35} &  (11,29,39,49,67,78,116) & $\frac{77}{10}$\\[1.5ex]

     \emph{16} &  (2,3,5,6,7,10,12) & $\geq\frac{3}{2}$
     & \emph{36} &  (13,22,55,76,97,110,152)& $\frac{117}{20}$\\[1.5ex]

     \emph{17} & (2,2n+1,2n+1,4n+1,6n+1,6n+3,8n+2) & $1$
     & \emph{37} &  (13,23,34,56,89,102,112) & $\frac{104}{15}$\\[1.5ex]

     \emph{18} &  (3,3,5,5,7,10,12) & $\geq\frac{7}{4}$
     & \emph{38} & (13,23,35,47,57,70,104) & $\frac{65}{8}$ \\[1.5ex]

     \emph{19} & (3,5,6,8,13,16,18) & $\geq\frac{5}{3}$
     & \emph{39} &   (13,23,35,57,79,92,114) & $\frac{91}{12}$ \\[1.5ex]

     \emph{20} &  (3,5,7,9,11,16,18) & $\geq\frac{14}{11}$
     & \emph{40} &  (14,19,25,32,45,64,70) & $\frac{28}{3}$ \\[1ex] 
     \hline
\end{tabular}  

\begin{tablenotes}\footnotesize
    \item[$a$] the defining equation of degree $6$ contains $yt$
    \item[$b$] the defining equation of degree $6$ contains no $yt$
    \end{tablenotes} 
\end{threeparttable}
\end{table}

\newpage

\begin{ack}
    The authors would like to thank Takashi Kishimoto and Masatomo Sawahara for their valuable suggestions. Their comments enable the authors to improve the result as well as the exposition. The authors also thank the referee for constructive comments that significantly improved the manuscript.
\end{ack}   

\begin{fund}
    The first author was supported by the National Research Foundation of Korea under Grant number RS-2025-00513064, while the third author was supported under Grant numbers RS-2025-00513064 and RS-2026-25506097. The second author was supported by the National Research Foundation of Korea through a grant funded by the Korean government (MSIT, RS-2024-00341743), and also partially by the Basic Science Research Program through the NRF, funded by the Ministry of Education (RS-2025-25415913).
\end{fund}

\bibliographystyle{amsplain}

\end{document}